\documentclass[12pt,reqno]{amsart}
\usepackage{amsmath,amsthm}
\usepackage[english]{babel}
\usepackage[margin=1.5in]{geometry}
\usepackage{amssymb}

\usepackage[applemac]{inputenc}
\usepackage{tikz}
\usetikzlibrary{arrows,automata}
\usepackage{amsfonts}
\usepackage{latexsym}
\usepackage{mathtools}
\usepackage{nccmath}
\usepackage{blkarray}
\usepackage{flexisym}
\usepackage{stmaryrd}
\usepackage{cases}
\usepackage{enumerate}
\usepackage{etoolbox}

\allowdisplaybreaks 
\theoremstyle{plain}
\newcommand{\C}{\mathbb C}
\newcommand{\lb}{\lambda}

\newcommand{\li}{\mathcal L}
\newcommand{\cp}{characteristic polynomial}
\newcommand*{\Keywords}[1]{\gdef\@Keywords{#1}}

\newtheorem{thm}{Theorem}[section]
\newtheorem{exmp}[thm]{Example}

\newtheorem{corollary}[thm]{Corollary}
	
\newtheorem{rmk}[thm]{Remark}	
\newtheorem{proposition}[thm]{Proposition}
\newtheorem{lemma}[thm]{Lemma}

\newtheorem{definition}[thm]{Definition}

\newtheorem{question}[thm]{Question}
\DeclareMathOperator{\ad}{ad}
\DeclareMathOperator{\aut}{Aut}
\DeclareMathOperator{\rank}{rank}
\DeclareMathSymbol{\mlq}{\mathord}{operators}{``}
\DeclareMathSymbol{\mrq}{\mathord}{operators}{`'}

\makeatletter\@addtoreset{equation}{section} \makeatother

\begin{document}

\title{Spectral invariants for finite dimensional Lie algebras}
\author[F. Azari Key and R. Yang]{Fatemeh Azari Key and Rongwei Yang}
\address{State University of New York At Albany \\ Department of Mathematics \\
Albany NY 12222\\
Dalian University of Technology \\ School of Mathematical Sciences \\ Dalian 116023}
\email{fa_azarikei@yahoo.com, fa.azari@mail.dlut.edu.cn, ryang@albany.edu}

\begin{abstract}
For a Lie algebra ${\mathcal L}$ with basis $\{x_1,x_2,\cdots,x_n\}$, its associated {\cp} 
$Q_{{\mathcal L}}(z)$ is the determinant of the linear pencil $z_0I+z_1\ad x_1+\cdots +z_n\ad x_n.$
This paper shows that $Q_{\mathcal L}$ is invariant under the automorphism group $\aut(\li).$ The zero variety and factorization of $Q_{\mathcal L}$ reflect the structure of $\li$. In the case $\li$ is solvable $Q_\li$ is known to be a product of linear factors. This fact gives rise to the definition of spectral matrix and the Poincar\'{e} polynomial for $\li$. Application is given to $1$-dimensional extensions of nilpotent Lie algebras.

\end{abstract}

\maketitle

\section{Introduction}

Given a square matrix $A\in M_k(\C)$, its characteristic polynomial $Q_A(z)=\det (A-zI)$ is a degree $k$ polynomial in one complex variable $z$ and it plays an important role in the study of $A$. However, in many situtations in mathematics and science we are concerned with several matrices $A_1,A_2,\cdots,A_n$ of the same size. Whether there exists a proper notion of characteristic polynomial for these several matrices is a difficult problem. One related study was dated back to the late 1890s when Dedekind and Frobenius considered the determinant of the linear pencil $A(z)=z_0I+z_1A_1+\cdots +z_nA_n$. To be precise, given a finite group $G=\{1, g_1,\cdots,g_n\}$,
its group algebra $\C[G]$ is a $(n+1)$-dimensional complex vector space with basis $G$. Multiplication on the left by an element $g\in G$ on any vector $h\in \C[G]$ results a permutation of coefficients of $h$ and hence it gives rise to a unitary matrix in $M_{n+1}(\C)$. In current day terminology this is the left regular representation of the group $G$ which we shall denote by $\lambda_G$ here. Dedekind studied the determinant 
\[Q_{\lambda_G}(z)=\det (z_0I+z_1\lambda_G(g_1)+\cdots +z_n\lambda_G(g_n))\]
for some groups and found it to have quite interesting factorizations. Likewise, for a general finite dimensional representations $\pi$ one may define the corresponding polynomial $Q_{\pi}(z)$. Later in 1896 Frobenius proved the following theorem (expressed in current terminology).\\

{\bf Theorem}(Dedekind-Frobenius \cite{De, Fr}). {\em If $G=\{1, g_1,\cdots,g_n\}$ is a finite group, then
\[Q_{\lambda_G}(z)=\prod_{\pi\in \hat{G}}Q_{\pi}^{d_\pi}(z),\]
where $\hat{G}$ is the set of equivalence classes of irreducible unitary representations of $G$ (called the unitary dual of $G$), and $d_\pi$ is the dimension of the representation $\pi$. Moreover, each such $Q_\pi$ is an irreducible polynomial.}\\

The polynomial $Q_\pi$ was later called the group determinant of $G$ and its study by Dedekind and Frobenius is indeed the starting point of group representation theory. For more information on this development, we refer readers to \cite{Cu,Cu2, Di21, Di75, FS}. However, the study of the group determinant has not been generalized to other algebraic settings. In 2009, the notion of projective spectrum was introduced in \cite{Ya} for general elements $A_1,A_2,\cdots,A_n$ in a unital Banach algebra based on the invertibility of the linear pencil $A(z)=z_1A_1+\cdots +z_nA_n$. Some follow-up work can be found in \cite{BCY,CY,CSZ,DY,GOY,GY,HWY,MQW17,SYZ}. Recent work related to characteristic polynomial of groups can be found in \cite{CST, HY18}. 

Paper \cite{HZ} by Zhang and Hu initiated the study of characteristic polynomials for Lie algebras. 
\begin{definition}
Consider a finite dimensional Lie algebra ${\mathcal L}$ with basis $S=\{x_1,x_2,\cdots,x_n\}$ and a linear representation $\rho$. The polynomial
\[Q_{{\mathcal L},\rho}(z):=\det (z_0I+z_1\rho(x_1)+\cdots +z_n\rho(x_n))\] is called the characteristic polynomial for ${\mathcal L}$ with respect to $\rho$ and the basis $S$. 
\end{definition}
In the case $\rho$ is the adjoint representation $\ad$, the characteristic polynomial will simply be called the characteristic polynomial of ${\mathcal L}$ and be denoted by $Q_{\mathcal L}$. Paper \cite{HZ} showed that $\li$ is solvable if and only if the characteristic polynomial $Q_{{\mathcal L},\rho}(z)$ is a product of linear factors for any representation $\rho$. This paper aims to make a general study on the characteristic polynomial of finite dimensional Lie algebras. The paper is organized as follows.

\tableofcontents

\noindent {\bf Acknowledgement}. Most part of this paper was completed when the authors were visiting the School of Mathematical Sciences at Tianjin Normal University. The authors thank the School for the hospitality and support. The first author would also like to thank Professor Yufeng Lu for his continuous encouragement and support.

\section{Characteristic polynomial and automorphism group}

Consider a Lie algebra ${\mathcal L}$ with basis $\{x_1,x_2,\cdots,x_n\}$. Then for $1\leq i, j\leq n$ using Einstein summation convention we can write $[x_i, x_j]=T_{i,j}^kx_k$. The set of complex numbers 
$\{T_{i,j}^k: 1\leq i, j, k\leq n\}$ are called the structure constants because they completely determine the structure of the Lie algebra ${\mathcal L}$. In particular, the matrix representation of the adjoint representation $\ad {x_i}$ is given by the matrix $T_i:=(T_{i,j}^k)_{1\leq k,j\leq n}$, whose rows are indexed by $k$ and columns are indexed by $j$. For instance, if ${\mathcal L}=span \Big\{x_1,x_2,x_3\Big \}$ is a $3$-dimensional Lie algebra then since
\begin{align*} 
&[x_1,x_1] = T^1_{11} x_1 + T^2_{11} x_2 + T^3_{11} x_3 \nonumber \\
&[x_1,x_2] = T^1_{12} x_1 + T^2_{12} x_2 + T^3_{12} x_3\nonumber \\
&[x_1,x_3] = T^1_{13} x_1 + T^2_{13} x_2 + T^3_{13} x_3 \nonumber,
\end{align*}
we have  \[T_1=\left(\begin{matrix}
 T^1_{11} &   T^1_{12} &  T^1_{13} \\
 T^2_{11} &   T^2_{12} &  T^2_{13} \\
 T^3_{11} &   T^3_{12} &  T^3_{13} 
\end{matrix}\right)=(T^k_{1j})_{3\times3} , 1\leq k,j \leq3.\]
For convenience, in many places of the paper we shall not distinguish $\ad {x_i}$ from its matrix representation $T_i$. For example, the characteristic polynomial 
$Q_{\mathcal L}(z)=\det (z_0I+z_1 T_1+\cdots +z_n T_n).$ Before we proceed with a general study, let us look at a particular example.

\begin{exmp}\label{ex:su2}
First, the $3$-dimensional simple Lie algebra $\frak{su}(2)$ plays an important role in mathematics and physics. It is spanned by the Pauli matrices
\[\sigma_1=\left(\begin{matrix}
0 & 1\\
1 & 0
\end{matrix}\right),\ \sigma_2=\left(\begin{matrix}
0 & -i\\
i & 0
\end{matrix}\right),\ \sigma_3=\left(\begin{matrix}
1 & 0 \\
0 & -1
\end{matrix}\right)\]
with the relations
\[ [\sigma_1,\sigma_2]=2i\sigma_3,\  [\sigma_2,\sigma_3]=2i\sigma_1,\  [\sigma_3,\sigma_1]=2i\sigma_2.\]
Then the matrix representations for $\ad{\sigma_1}, \ad{\sigma_1}$ and $\ad{\sigma_3}$ are
\[T_1=\left(\begin{matrix}
0 & 0 & 0\\
0 & 0 & -2i\\
0 & 2i & 0
\end{matrix}\right),\ T_2=\left(\begin{matrix}
0 & 0 & 2i\\
0 & 0 & 0\\
-2i & 0 & 0
\end{matrix}\right), \ T_3=\left(\begin{matrix}
0 & -2i & 0\\
2i & 0 & 0\\
0 & 0 & 0
\end{matrix}\right),\]
respectively. One computes that its characteristic polynomial 
\begin{align*}
Q(z)&=\det (z_0I+z_1T_1+z_2T_2+z_3T_3)\\
&=z_0\left(z_0^2-4(z_1^2+z_2^2+z_3^2)\right).
\end{align*}
\end{exmp}

\subsection{Invariance of characteristic polynomial}
We start the general study with a lemma which describes how characteristic polynomial changes with respect to isomorphism. For any $z=(z_0,z_1,\cdots,z_n)$ we can write $z=(z_0, z')$ where $z'$ stands for the row vector $(z_1,z_2,\cdots,z_n)$. Then every matrix $B\in GL_n$ gives rise to a change of variables $(z_0, z')$ by $(z_0, z'B)$. For convenience, we will write $(z_0, z'B)$ simply as $zB$.
Also, we let $GL_n$ denote the general linear group of invertible $n\times n$ complex matrices.

\begin{lemma}\label{lem:iso}
If two $n$-dimensional Lie algebras $\li_1$ and $\li_2$ are isomorphic, then there exists matrix $B\in GL_n$ such that $Q_{\li_1}(z)=Q_{\li_2}(zB).$
\end{lemma}
\begin{proof}
We first describe how the {\cp} varies with respect to a change of basis. Let $\{x_1,x_2, \hdots,x_n\}$ and $\{\hat{x}_1,\hat{x}_2,\hdots, \hat{x}_n \}$ be two bases for the Lie algebra $\li$, and assume
$\hat{x}_{i}=b_{ij}x_{j}$. Then the transformation matrix is $B=(b_{ij})_{n\times n}\in GL_n$. Suppose the structure constants with respect to the two bases are given by the equations
\begin{equation}\label{eq:sc}
[x_{i},x_{j}]=T^{k}_{ij} x_{k}, \ \ \ [\hat{x}_i,\hat{x}_j]= \hat{T}^k_{ij}\hat{x}_k,
\end{equation}
respectively, then one verifies that
\begin{align*}
[\hat{x}_i,\hat{x}_j]&=[\Sigma_\alpha b_{i\alpha}x_{\alpha},\Sigma_\beta b_{j\beta}x_{\beta}] \\
&=\Sigma_{\alpha,\beta} b_{i\alpha} b_{j\beta} [x_{\alpha},x_{\beta}] \nonumber \\
&=\Sigma_{\alpha,\beta} b_{i\alpha} b_{j\beta} \Sigma_k {T}^{k}_{\alpha\beta} x_{k}. 
\end{align*}
Hence by (\ref{eq:sc}) one has
\[\Sigma_{\alpha,\beta} b_{i\alpha} b_{j\beta} \Sigma_k {T}^{k}_{\alpha\beta} x_{k}=\Sigma_k  \hat{T}^k_{ij}  \hat{x}_k =\Sigma_k \Sigma_l  \hat{T}^k_{ij} b_{kl}x_{l}.\]
Equating the coefficients of $x_l$, one has
\[\Sigma_{\alpha,\beta} b_{i\alpha} b_{j\beta} {T}^l_{\alpha,\beta}=\Sigma_{k}\hat{T}^k_{ij} b_{kl}\nonumber,\]
which one writes in matrix form as
\begin{equation}\label{eq:mtrans}
\Sigma_{\alpha} b_{i\alpha}{{T}_{\alpha}}B^t=B^t \hat T_{i},\ \ 1\leq i\leq n, 
\end{equation}
where $B^{t}$ stands for the transpose of $B$, and $T_i, \hat{T}_i$ are as defined before with respect to the two bases. It follows that
\begin{align*}
\hat Q(z)=&det(z_{0}I+\sum_{j=1}^n z_{j}\hat T_{j}) \nonumber\\
=&\frac{det(z_{0}B^{t}+ z_{1}B^{t}\hat T_{1} + \cdots + z_{n}B^{t}\hat T_{n})}{det(B^t)}\nonumber\\
=&\frac{1}{det(B^{t})}det(z_{0}B^{t}+z_{1}\Sigma_{\alpha}b_{1\alpha}{T}_{\alpha}B^{t}+z_{2}\Sigma_{\alpha}b_{2\alpha} T_{\alpha}B^{t}+\cdots
+z_{n}\Sigma_{\alpha}b_{n\alpha} T_{\alpha}B^{t})\nonumber\\
=&det(z_{0}I+z_{1}\Sigma_{\alpha}b_{1\alpha} T_{\alpha}+\cdots+z_{n}\Sigma_{\alpha}b_{n\alpha}T_{\alpha})\nonumber \\
=&det\Bigg(z_{0}I+(z_{1}b_{11}+\cdots+z_{n}b_{n1})T_{1}+\cdots +(z_{1}b_{1n}+\cdots+z_{n}b_{nn}) T_{n}\Bigg)\nonumber.
\end{align*}
 Now setting $w_0=z_0$ and  $w_{i}=b_{1i}z_{1}+b_{2i}z_{2}+\cdots+b_{ni}z_{n}, \ 1\leq i\leq n$, we have 
\[(w_1,w_2,\hdots, w_n)=(z_1,z_2,\hdots, z_n)B,\]
and consequently,
\begin{equation}\label{eq:trans}
\hat Q(z)=\det(z_{0}I +\sum_{j=1}^n  w_{j}T_{j})=Q(zB).
\end{equation} 

Now suppose $\li_1=span\Big \{x_1,x_2\hdots,x_n \Big \}$ and $\li_2=span\Big \{y_1,y_2,\hdots,y_n \Big \}$ are two Lie algebras and consider their characteristic polynomials
\begin{align*}
 Q_{\li_1}(z)=\det \big(z_0 I+\sum_{j=1}^nz_j \ad {x_j}\big),\ \ 
Q_{\li_2}(w)=\det \big(w_0 I+\sum_{j=1}^nw_j \ad {y_j}\big).
\end{align*}
Assume $\phi: \li_1\to \li_2$ is an isomorphism and $\hat{x}_i= \phi(x_i),\ 1\leq i\leq n$.  Then
\[\hat{T}_{i,j}^k\hat{x}_k=[\hat{x}_i, \hat{x}_j]=\phi([x_i, x_j])=T_{i,j}^k\hat{x}_k,\]
which implies $T_j=\hat{T}_j,\ 1\leq j\leq n$, and therefore
\[Q_{\li_1}(z)=\det\left(z_0I + \sum_{j=1}^n z_j\hat{T}_j\right):=\hat{Q}(z).\]
Now suppose $\hat{x}_{i}=b_{ij}y_{j}$, then the lemma follows from (\ref{eq:trans}).

\end{proof}

Let $\aut(\li)$ denote the group of automorphisms on a Lie algebra $\li$. Then for a fixed basis $\{x_1,x_2,\cdots,x_n\}$, every $\phi\in \aut(\li)$ has a matrix representation $B_\phi=(b_{ij})$ given by $\phi(x_i)=b_{ij}x_j ,\ 1\leq i,j\leq n$. Lemma \ref{lem:iso} then implies the following theorem. 

 \begin{thm}\label{thm:aut}
Let $\li$ be a finite dimensional Lie algebra. Then the characteristic polynomial $Q_{\li}(z)$ is invariant under $\aut(\li)$ in the sense that for every $\phi\in \aut(\li)$ one has $Q_{\li}(z)=Q_{\li}(zB_\phi)$.
\end{thm}

Theorem \ref{thm:aut} helps to reveal more information about the automorphism group.
Recall that the orthogonal group $O(n)=\{A\in GL_n: A^tA=I\}$. With the identification of $\phi\in \aut(\li)$ and its matrix representation $B_\phi=(b_{ij})$, we can state the following corollary.

\begin{corollary}
Let $\li$ be a $3$-dimensional simple Lie algebra, then $\aut(\li)$ is a subgroup of $O(3)$.
\end{corollary}
\begin{proof}
The characteristic polynomial of $\frak{su}(2)$ is calculated in Example \ref{ex:su2}.
For any $\phi\in \aut(\frak{su}(2))$ its matrix representation $B_\phi$ preserves $Q(z)$ in the sense that $Q(zB_\phi)=Q(z)$, i.e., the change of variables $(z_1, z_2, z_3)$ by $(z_1, z_2, z_3)B_\phi$ preserves the quadratic form $z_1^2+z_2^2+z_3^2$, which concludes that $B_\phi\in O(3)$. Since up to isomorphism there is only one complex $3$-dimensional simple Lie algebra, we obtained the lemma.
\end{proof}

Some other facts follow from the proof of Lemma \ref{lem:iso}.
\begin{corollary}
Consider Lie algebra $\li=span \Big\{x_{1},x_{2},\hdots,x_{n} \Big\}$. Suppose $B=(b_{ij}) \in GL_n$, and $\hat x_{i}=b_{ij}x_{j}$.Then
\begin{align*}
B
\left(\begin{matrix}
t{r}T_{1} \\
t{r}T_{2}\\
\vdots \\
t{r}T_{n}
\end{matrix}\right)=
\left(\begin{matrix}
t{r}\hat{T}{_1} \\
t{r}\hat{T}{_2}\\
\vdots \\
t{r}\hat{T}{_n}
\end{matrix}\right).
\end{align*}
\end{corollary}
\begin{proof}
Writing (\ref{eq:mtrans}) as $b_{i\alpha}T_{\alpha}=B^t \hat T_{i} (B^t)^{-1}$ and applying the trace to both sides, one obtains
\[tr( b_{i\alpha}T_{\alpha})=tr(B^t \hat T_{i} (B^t)^{-1})=tr \hat T_{i},\]
and it follows that $ b_{i\alpha}tr T_{\alpha}=tr \hat T_{i},  \forall {1\leq i\leq n}.$
\end{proof}

\begin{rmk} In the case $\phi\in \aut({\li})$, one has $T_i=\hat{T}_i$ for each $i$ and  therefore the vector 
$\left(t{r}T_{1}, t{r}T_{2},\cdots,t{r}T_{n}\right)^t$, if it is nonzero, is an eigenvector of $B_\phi$. \end{rmk}

\begin{corollary}
Suppose $\li=span  \Big\{x_{1},x_{2}\hdots,x_{n} \Big\}$ and $\hat \li=span \Big\{\hat x_{1},\hat x_{2}\hdots,\hat x_{n} \Big\}$ are two Lie algebras. If $t{r} (\hat T_{j})=0$ for all $1\leq j \leq n$ but $tr(T_{i}) \neq 0$ for some $i$, then $\li$ and $\hat \li$ are not isomorphic.
\end{corollary}

\subsection{Automorphism group of $\frak{sl}(2)$}

A finer study about unitary elements in $\aut(\li)$ can be made in the case $\li=\frak{sl}(2)$.

\begin{exmp}\label{ex:sl}
Consider the $3$-dimensional simple Lie algebra $\frak{sl}(2)$ with basis $\{H,X,Y\}$ such that 
\[[H,X]=2X ,\,\,[H,Y]=-2Y,\ \ [X,Y]=H.\]
Hence the matrix representations for $\ad H, \ad X$ and $\ad Y$ are
\begin{align*}
T_{H}=\left(\begin{matrix}
0 & 0 & 0\\
0 & 2 & 0\\
0 & 0 & -2
\end{matrix}\right)
,\,\,T_{X}=\left(\begin{matrix}
0 & 0 & 1\\
-2 & 0 & 0\\
0 & 0 & 0
\end{matrix}\right)
,\,\,T_{Y}=\left(\begin{matrix}
0 & -1 & 0\\
0 & 0 & 0\\
2 & 0 & 0
\end{matrix}\right),
\end{align*}
respectiely. And therefore the characteristic polynomial of $\frak{sl}(2)$ is 
\begin{align*}
Q(z)&=\det (z_0I+z_1T_H+z_2 T_X+z_3T_Y)\\
&=\det \left(\begin{matrix}
z_0 & -z_3 & z_2\\
-2z_2 & z_0+2z_1 & 0\\
2z_3 & 0 & z_0-2z_1
\end{matrix}\right)\\
&=z_{0}\left(z_{0}^2-4(z_{1}^2 + z_{2}z_{3})\right).
\end{align*}
\end{exmp}

The irreducible representations of $\frak{sl}(2)$ have been fully described in \cite{FH,Hu}. And the characteristic polynomials with respect to these representations were recently determined in \cite{CCD,HZ}, where it is shown that if $\pi: \frak{sl}(2)\to \frak{gl}({m+1})$ is an irreducible representation (which is unique) then the \cp
\begin{equation}
\label{eq:csl}
Q_\pi(z)= \begin{dcases}
      \prod_{j=0}^{(m-1)/2 }\left(z_0^2-(m-2j)^2(z_1^2+z_2z_3)\right), & m\ \text{odd};\\
       z_0\prod_{j=0}^{(m/2)-1 }\left(z_0^2-(m-2j)^2(z_1^2+z_2z_3)\right), & m\ \text{even.}
    \end{dcases}
\end{equation}
 One observes that the variables $z_1, z_2, z_3$ appear only in the quadratic form 
$z_{1}^2 + z_{2}z_{3}$ in $Q_\pi(z)$ . Since by Theorem \ref{thm:aut} the {\cp} $Q(z)$ in Example \ref{ex:sl} is invariant under $\aut(\frak{sl}(2))$, the above $Q_\pi(z)$ is invariant under $\aut(\frak{sl}(2))$. Further, since by Weyl's theorem every representation of a semi-simple Lie algebra is the direct sum of irreducible representations, foregoing observations lead to the following fact.

\begin{corollary}
Let $\pi$ be any linear representation of $\frak{sl}(2)$. Then the {\cp} $Q_\pi(z)$ is invariant under $\aut(\frak{sl}(2))$.
\end{corollary}

The invariance of $z_{1}^2 + z_{2}z_{3}$ under $\aut(\frak{sl}(2))$ also reveals more details about the latter. Let ${\mathbb T}$ denote the unit circle.

\begin{corollary}
The subgroup of unitary elements in $\aut\left(\frak{sl}(2,\mathbb{C})\right)$ is 
\[\left\{\left(\begin{matrix}
1 & 0 & 0\\
0 & \alpha & 0\\
0 & 0 & \overline{\alpha}
\end{matrix}\right),\ \ 
\left(\begin{matrix}
-1 & 0 & 0\\
0 & 0 & \beta\\
0 & \overline{\beta} & 0
\end{matrix}\right),\ \alpha, \beta\in {\mathbb T}\right\}.\]
\end{corollary}
\begin{proof}
The characteristic polynomial of $\frak{sl}(2)$ is computed in Example \ref{ex:sl} as 
$Q(z)=z_{0}(z_{0}^2-4(z_{1}^2 + z_{2}z_{3})).$ Write
\begin{align*}
&z_{1}^2+z_{2}z_{3}=(z_{1},z_{2},z_{3})M\left(\begin{matrix}
z_{1}\\
z_{2}\\
z_{3}
\end{matrix}\right),\,\,\textnormal{where}\,\,M=\left(\begin{matrix}
1 & 0 & 0\\
0 & 0 & \frac{1}{2}\\
0 & \frac{1}{2} & 0
\end{matrix}\right).
\end{align*}
Suppose $\phi$ is a unitary element in $\aut(\frak{sl}(2))$ (for the moment we assume such element exists). Here again for convenience we shall not distinguish $\phi$ from its matrix representation $B_\phi$. Then it follows from Theorem \ref{thm:aut} that $\phi M \phi^t=M$, where $\phi=(\phi_{ij})_{3 \times 3}$, i.e., 
\begin{align*}
&\left(\begin{matrix}
\phi_{11} & \phi_{12} & \phi_{13}\\
\phi_{21} & \phi_{22} & \phi_{23}\\
\phi_{31} & \phi_{32} & \phi_{33}
\end{matrix}\right)
\left(\begin{matrix}
1 & 0 & 0\\
0 & 0 & \frac{1}{2}\\
0 & \frac{1}{2} & 0
\end{matrix}\right)
\left(\begin{matrix}
\phi_{11} & \phi_{21} & \phi_{31}\\
\phi_{12} & \phi_{22} & \phi_{32}\\
\phi_{13} & \phi_{23} & \phi_{33}
\end{matrix}\right)=
\left(\begin{matrix}
1 & 0 & 0\\
0 & 0 & \frac{1}{2}\\
0 & \frac{1}{2} & 0
\end{matrix}\right).\\
\end{align*}
This results in the following equations :
\begin{numcases}{}
\phi_{11}^2 + \phi_{13} \phi_{12}=1\\
\phi_{21}^2 + \phi_{23} \phi_{22}=0\\
\phi_{31}^2 + \phi_{33} \phi_{32}=0\\
\phi_{21}\phi_{11} + \frac{1}{2} \phi_{23}\phi_{12}+\frac{1}{2} \phi_{22}\phi_{13}=0\\
\phi_{31}\phi_{11} + \frac{1}{2} \phi_{33}\phi_{12}+\frac{1}{2} \phi_{32}\phi_{13}=0\\
\phi_{21}\phi_{31} + \frac{1}{2} \phi_{23}\phi_{32}+\frac{1}{2} \phi_{22}\phi_{33}=\frac{1}{2}
\end{numcases}
Also since $\phi$ is a unitary matrix we have
\begin{numcases}{}
 |\phi_{11} \vert^2 \,+\,\vert \phi_{12} \vert^2\,+\,\vert \phi_{13} \vert^2 =1\\
 |\phi_{21} \vert^2\,+\,\vert \phi_{22} \vert^2\,+\,\vert \phi_{23} \vert^2 =1\\	
 |\phi_{31} \vert^2\,+\,\vert \phi_{32} \vert^2\,+\,\vert \phi_{33} \vert^2 =1
\end{numcases}\\
From (2.5) and (2.11) and the fact that
\begin{align*}
1 &=\vert \phi_{11}^2 +\,\phi_{12} \phi_{13} \vert \,\, \leq \,\, \vert \phi_{11} \vert^2 + \, \vert \phi_{12} \phi_{13}\, \vert \,\,\leq \vert \phi_{11} \vert ^2 + 2\,\vert \phi_{12} \phi_{13} \vert\\
&\leq \vert \phi_{11} \vert ^2 + \vert \phi_{12} \vert ^2 + \vert \phi_{13} \vert ^2=1
\end{align*}
we have $|\phi_{12}\phi_{13}|=2|\phi_{12}\phi_{13}|$ which implies $\phi_{12}\phi_{13}=0$. Hence by  (2.5) we have $\phi_{11}=\pm 1$, and consequently $\phi_{12}=\phi_{13}=\phi_{21}=\phi_{31}=0$ since $\phi$ is unitary. Therefore,
\begin{flalign}
\phi=\left(\begin{matrix}
\pm 1 & 0 & 0\\
0 & \phi_{22} & \phi_{23}\\
0 & \phi_{32} & \phi_{33}
\end{matrix}\right).\nonumber
\end{flalign}
By (2.12) and (2.13) we get
\begin{numcases}{}
 \vert \phi_{22} \vert^2 \,+\,\vert \phi_{23} \vert^2\,=1 \nonumber \\
\vert \phi_{32} \vert^2\,+\,\vert \phi_{33} \vert^2\,=1
\end{numcases}
Now (2.10) gives $\phi_{22}\,\,\phi_{33} \,+\, \phi_{23}\phi_{32}\,=1,$
i.e., the inner product 
\begin{equation}
\Bigg \langle (\phi_{22},\phi_{23}),(\bar \phi_{33},\bar \phi_{32}) \Bigg \rangle=1,
\end{equation}
 which imply $\phi_{22}=\bar \phi_{33}$ and $\phi_{23}=\bar \phi_{32}.$
Moreover, it follows from the equalities (2.6) that $\phi_{23}\phi_{22}=0$.  
In the case $\phi_{23}=0$, Equality (2.15) gives $ \phi_{22}\phi_{33}=1;$ 
in the case $\phi_{22}=0$, we have $\phi_{23}\phi_{32}=1$. 
Therefore $\phi$ has the following possible choices:
\begin{flalign*}
\phi=\left(\begin{matrix}
\pm 1 & 0 & 0\\
0 & \alpha & 0\\
0 & 0 & \overline{\alpha}
\end{matrix}\right)\ 
\text{and}\ 
\phi=\left(\begin{matrix}
\pm 1 & 0 & 0\\
0 & 0 & \beta\\
0 & \overline{\beta} & 0
\end{matrix}\right),
\end{flalign*}
where $\alpha$ and $\beta$ are unimodular.

From the other side, being an automorphism of $\frak{sl}(2)$ the transformation $\phi$ needs to satisfy additional conditions on its basis, namely
\begin{enumerate}[1)]
\item $\phi([H,X])=[\phi(H),\phi(X)],$
\item $\phi([H,Y])=[\phi(H),\phi(Y)],$
\item $\phi([X,Y])=[\phi(X),\phi(Y)].$
\end{enumerate}
A quick check of these conditions establishes the corollary.
\end{proof}

\section{Semidirect sum}

In this section we consider the characteristic polynomial of a semidirect sum of two Lie algebras. 
Quite a few definitions need to be recalled to proceed. 

\begin{definition}
A derivation $D$ of a given Lie algebra $\li$ over a field $\mathbb{F}$,\,\,is a linear map $D: \li \rightarrow \li$ satisfying the Leibnitz law with respect to the Lie bracket:
\[D([x,y])=[D(x),y]+[x,D(y)] , \forall x,y \in \li.\]
\end{definition}
\begin{definition}
A derivation $D$ of $\li$ is called inner if there exists an element $z \in \li$ such that $D=\ad z$, i.e., $D(x)=[z,x],\ \forall x\in \li$.
\end{definition}
The set of of all derivations of $\li$, denoted by $Der(\li)$, is a vector space which can be given a Lie algebra structure. The set $Inn(\li)$ of inner derivations form an ideal of $Der(\li).$

\begin{definition}
Let $\li_{1}$ and $\li_{2}$ be Lie algebras over the field $\mathbb{F}$ and $\tau: \li_{1} \to Der(\li_{2})$ be a fixed homomorphism. The semidirect sum
$\li_1\niplus_\tau \li_2$ is the direct sum of their underlying vector spaces $\li_1\oplus \li_2$ with the Lie bracket defined as
\begin{align*} 
[(x_{1},y_{1}),(x_{2},y_{2})]=\left([x_{1},x_{2}],\tau (x_{1})(y_{2})-\tau(x_{2})(y_{1})+[y_{1},y_{2}]\right).
\end{align*}
\end{definition}
Clearly, the semidirect sum $\li_1\niplus_\tau \li_2$ varies with the choice of $\tau$.
In general $\li_{1}$ is a subalgebra and $\li_{2}$ is an ideal in $\li_1\niplus_\tau \li_2$.
In the special case when $\tau$ is identically zero, the semidirect sum turns into the direct sum $\li_{1} \oplus \li_{2}$. A good treatise on semidirect sum can be found in \cite{Hu} and \cite{Os}.

Now suppose $\li=\li' \niplus_\tau \li''$ is a semidirect sum of two Lie algebras.
Assume $\dim \li'=k>0$, $\dim \li''=n-k \ge 0$, and let $\Big\{x_{1},x_{2},\hdots,x_{k}  \Big\}$ and $\Big\{x_{k+1},x_{k+2},\hdots,x_{n}  \Big\}$ be bases for $\li'$ and $\li''$, respectively. Then on  $\li' \niplus_\tau \li''$ the adjoint representation $\ad x_{i}$ can be represented by the block matrices 
\[T_{i}=\left(\begin{matrix}
T'_{i} & 0\\
0 & \tau(x_i)
\end{matrix}\right),\ 1 \leq i \leq k; \ \ \ T_j=\left(\begin{matrix}
0 & 0\\
X_{j} & T''_{j}
\end{matrix}\right),\ k+1 \leq j \leq n,\]\\
where $T'_{i}$ is the $k \times k$ matrix for the adjoint representation of $x_{i}$ on $\li'$, $T''_{j}$ is the $(n-k)\times (n-k)$ matrix for the adjoint representation of $x_{j}$ on $\li''$,
$\tau(x_i)\in Der(\li'')$ is an $(n-k)\times (n-k)$ matrix, and
$X_{j}$ is the $(n-k)\times k$ matrix for the adjoint representation of $x_{j}$ on $\li'$. Hence
\begin{align*}
&z_{0}I + \sum_{j=1}^n z_{j}T_{j}\\
&=\left(\begin{matrix}
z_{0}I_{k} + \sum_{j=1}^k z_{j}T'_{j} & 0\\
\sum_{j=k+1}^n z_{j}X_{j} & z_{0}I_{n-k} + \sum_{j=k+1}^n z_{j}T''_{j} + \sum_{i=1}^kz_{i}\tau(x_{i})
\end{matrix}\right),
\end{align*}
which implies 
\begin{equation}\label{eq:prod}
Q_\li(z)=Q_{\li'}(z)\det\left(z_{0}I_{n-k} + \sum_{j=k+1}^n z_{j} T''_{j} + \sum_{i=1}^k z_{i} \tau(x_{i})\right).
\end{equation}

Two immediate observations are made as follows.
\begin{rmk}\label{rm:prod}
One observes that if $\li''$ is abelian, then $T''_{j}=0$ for $k+1 \leq j \leq n$ and hence $Q_{\li}(z)=Q_{\li'}(z)\det \left(z_{0}I +\sum_{i=1}^k z_{i}\tau(x_{i})\right)$, where the second factor is the {\cp} for the matrices $\tau(x_{i}), 1\leq i\leq k$. On the other hand, if $\tau=0$ then $\li=\li'\oplus \li''$ and Equation \ref{eq:prod} implies $Q_{\li}(z)=Q_{\li'}(z)Q_{\li''}(z).$
\end{rmk}

\begin{definition}
A Lie algebra $\li$ is said to be simple if it contains no proper ideal, and it is said to be semisimple if it is a direct sum of simple Lie algebras.
\end{definition}

Consider a Lie algebra $\li$ with basis $\{x_1, ..., x_n\}$ and let \[L(z)=z_1\ad {x_1}+\cdots +z_n\ad {x_n}.\] It is not hard to see that $L(z)$ has a nontrivial kernel which in particular contains the element $z_1{x_1}+\cdots +z_nx_n$. The following lemma is thus immediate.
\begin{lemma}\label{lem:z0}
Let $\li$ be a finite dimensional Lie algebra. Then its {\cp} $Q_\li(z)$ has a factor $z_0$.
\end{lemma}
Hence it makes better sense to consider the polynomial $Q_\li^*(z)=\frac{Q_{\li}(z)}{z_{0}}$.
\begin{corollary}\label{thm:irr}
If $\li$ is a Lie algebra of dimension $n\geq 3$ and $Q_\li^*(z)$ is irreducible, then $\li$ is simple.
\end{corollary}
\begin{proof}
By Levi's decompostion every Lie algebra $\li$ can be represented as the semidirect sum $\li=\li' \niplus_{\ad} \li''$, where $\li'$ is a semisimple subalgebra of $\li$ and $\li''$ is the radical of $\li$, i.e., 
the maximal solvable ideal in $\li$.
Denoting the determinant in Equation \ref{eq:prod} by $R(z)$, we then have $Q^*_\li(z)=Q^*_{\li'}(z)R(z)$. Now since $Q_\li^*(z)$ is irreducible, either $Q^*_{\li'}$ is constant or $R(z)$ is constant.
In the former case, we have $\li=\li''$ which is solvable and hence $Q_\li(z)$ is a product of $n$ linear factors by \cite{HZ}. This contradicts with the irreducibility of $Q_\li^*(z)$. On the other hand, if $R(z)$ is constant then $\li=\li'$ which is semisimple. If $\li$ is a direct sum of more than one simple Lie algebras then $Q_\li(z)$ is the product of the {\cp} of these simple Lie algebras by Remark \ref{rm:prod} which again contradicts with the irreducibility of $Q^*_\li(z)$. This concludes that $\li$ is simple.
\end{proof}

Since finite dimensional complex simple Lie algebras have been fully characterized (\cite{Hu}), it is tempting to ask the following question.
\begin{question}
For which type of finite dimensional complex simple Lie algebras $\li$ is the polynomial $Q^*_\li(z)$ irreducible?
\end{question}

\section{Solvable Lie algebra and spectral matrix}

Given a Lie algebra $\li$, its derived algebra $[\li, \li]$ is the linear span of all elements $[s, t], s, t\in \li$. We set $\li^0=\li=\li_0$, $\li^k=[\li^{k-1}, \li^{k-1}]$ and $\li_k=[\li, \li_{k-1}]$. The algebra $\li$ is said to be solvable if $\li^k=\{0\}$ for some integer $k\geq 1$, and it is said to be nilpotent if $\li_k=\{0\}$ for some integer $k\geq 1$. Clearly, since $\li^k\subset \li_k$, every nilpotent Lie algebra is solvable. The largest nilpotent ideal in $\li$ is called the nilradical of $\li$ and is denoted by $nil(\li)$.
Engel's theorem asserts that a finite dimensional Lie algebra $\li$ is nilpotent if and only if for every element $x\in \li$ the matrix $\ad x$ is nilpotent.
It follows from Engel's theorem that the nilradical $nil(\li)$ is the set of all elements $x\in \li$ whose adjoint representation $\ad x$ is a nilpotent matrix, i.e., the spectrum $\sigma(\ad x)=\{0\}$. Equivalently, there exists integer $s\geq 1$ such that $(\ad x)^s=0$. The smallest such integer $s$ for which $(\ad x)^s=0$ for all $x\in \li$ is called the nilindex of $\li$. An estimate of the dimension of $nil(\li)$ is $\dim nil(\li)\geq \frac{1}{2}\dim \li$, and a finer estimate is given in \cite{Sn}. A good reference on nilpotent and nilpotent Lie algebras is \cite{GK}.

 In this section, we study how the {\cp} may reveal the structure of a nilpotent or solvable Lie algebra. In particular, we will look at the classification problem for solvable Lie algebras from the point of view of their {\cp}s. Recall here the Hu-Zhang theorem that a finite dimensional Lie algebra is solvable if and only if its {\cp} with respect to any linear representation is a product of linear factors. The following lemma is simple but quite useful.

\begin{lemma}\label{lem:linear}
Let $A_1,A_2,\cdots, A_n$ be $k\times k$ matrices and suppose their {\cp} $\det (z_0I+\sum_jz_jA_j)$ has a linear factor $z_0+\lambda_1z_1+\cdots +\lambda_nz_n$. Then $\lambda_j$ is an eigenvalue of $A_j$ for each $j$.
\end{lemma}
Proof is almost trivial, for example to check $\lambda_1$ is an eigenvalue of $A_1$ one may simply set $z_2=\cdots =z_n=0$. Then by the assumption $\det (z_0I+z_1A_1)$ has a factor $z_0+\lambda_1z_1$ and the conclusion follows.

Now consider a solvable Lie algebra $\li$ with a fixed basis $\{x_1,x_2,\cdots,x_n\}$ and let $T_i=\ad {x_i}$.
Then by Hu-Zhang theorem one may write
 \begin{equation}\label{eq:factors}
Q_{\li}(z)=\prod_{j=1}^{n}\left(z_{0} + \sum_{i=1}^n \lambda_{ij}z_{i}\right).
\end{equation}
In view of Lemma \ref{lem:linear}, one has $\sigma(T_{i})=\Big\{ \lambda_{ij}\,\big \vert \, 1\leq j\leq n \Big\}$ for each $i$. 

\begin{definition}
Let $\li$ be a solvable Lie algebra with basis $S=\{x_1,x_2,\cdots,x_n\}$ and $Q_{\li}(z)$ be its {\cp} as in (\ref{eq:factors}). The spectral matrix for $\li$ with respect to the basis $S$ is defined as 
 $\lambda_\li =(\lambda_{ij})_{n\times n} ,1\leq i,j \leq n$.
\end{definition}
We shall write $\lb_\li$ simply as $\lb$ whenever there is no risk of confusion. Since the $i$-th row of $\lambda$ is the set of eigenvalues of $T_i$, counting multiplicity, the order of the $\lambda$'s rows is in line with the order of $T_1,T_2,\cdots,T_n$. Moreover, the order of factors in the factorization (\ref{eq:factors}) determines the order of $\lambda_{i1},\lambda_{i2},\hdots,\lambda_{in}$ for every $1 \leq i \leq n$. In other words it determines the order of columns in the spectral matrix. Since $z_0$ is always a factor of $Q_\li$ by Lemma \ref{lem:z0}, we take it to be the first factor, e.g., $\lambda_{i1}=0$ for all $1\leq i\leq n$.

\begin{lemma}\label{lem:lb}
Let $\{x_1,x_2,\cdots,x_n\}$ and $\{\hat{x}_1,\hat{x}_2,\cdots,\hat{x}_n\}$ be two bases for a solvable Lie algebra $\li$ and assume $\lb$ and $\hat{\lb}$ are the corresponding spectral matrices.
Then $\lb=B\hat{\lambda}$, where $B=(b_{ij})$ is such that $x_{i}=b_{ij}\hat{x}_{j}$ for $1\leq i, j\leq n$. 
\end{lemma}
\begin{proof}
For simplicity, we let $Q(z)$ and $\hat{Q}(z)$ denote the {\cp}s with respect to the two bases. Then by (\ref{eq:trans}), we have $Q(z)=\hat{Q}(w)={\hat{Q}}(zB)$, where $w_0=z_0$ and  $w_{k}=\sum_{i=1}^n z_{i}b_{ik} , 1 \leq k \leq n$.
In other words, we have
\begin{flalign*}
\prod_{j=1}^n(z_{0} + \sum_{i} \lambda_{ij}z_{i})&=\prod ^{n}_{j=1}(z_{0} + \sum_{k}\hat{\lambda}_{kj}w_{k})\\
&= \prod ^{n}_{j=1}(z_{0} + \sum_{i,k}\hat{\lambda}_{kj}z_{i}b_{ik})
= \prod ^{n}_{j=1}\big(z_{0} + (\sum_{i,k}b_{ik}\hat{\lambda}_{kj})z_{i}\big),
\end{flalign*}
which shows $\lambda_{ij}=b_{ik} \hat{\lambda}_{kj}$,
i.e., $\lambda=B\hat{{\lambda}}$.
\end{proof}
 
In the case $\li$ and $\hat{\li}$ are isomorphic solvable Lie algebras of dimension $n$ with corresponding spectral matrices $\lb$ and $\hat{\lb}$, respectively. Then an argument similar to the ending part of the proof to Lemma \ref{lem:iso} will show that there exsits $B\in GL_n$ such that
$\lb=B\hat{\lambda}$. In particular, this suggests that the rank of $\lb$ is an invariant under the isomorphism of solvable Lie algebras. It is thus interesting to determine this rank by the structure of $\li$. 
\begin{definition}
Let $\li$ be a $n$-dimensional Lie algebra and $1\leq k\leq n$. The elements $f_{1},f_{2},\cdots,f_{k}$ in $\li$ are said to be nil-independent if there exist no scalars $c_{1},c_{2},\hdots, c_{k}$, not all $0$,
 such that $\sum_{i=1}^k c_{i}f_{i} \in nil(\li)$.
\end{definition}
In other words $f_{1},f_{2},\hdots,f_{n}$ are nil-independent if and only if $[f_{i}], 1\leq i \leq k$ are linearly independent in the quotient algebra $\li/nil(\li).$

\begin{proposition}\label{prop:rank}
Given a solvable Lie algebra $\li$ we have \[\rank \lambda=\dim \li/nil(\li).\]
\end{proposition}
\begin{proof}
Assume that $\li$ has dimension $n$ and its nilradical $nil(\li)$ has a basis $\Big\{x_{1},x_{2},\hdots,x_{k}\Big\}$, where $k\leq n$. Assume further that $x_{k+1},x_{k+2},\hdots,x_{n}$ are elements in $\li$ such that $\{[x_{k+1}],[x_{k+2}],\hdots,[x_{n}]\}$ is a basis for the quotient  $\li/nil(\li)$. Then clearly $\{x_1, ..., x_n\}$ is a basis for $\li$ and $T_i=\ad {x_i}, 1\leq i\leq k$ are nilpotent matrices by Engel's theorem. This means that the first $k$ rows in the spectral matrix $\lb$ are $0$s and hence $\rank \lb\leq n-k$. Suppose $t=\rank \lambda< n-k$. Then there exists an invertible $n \times n$ matrix $A=(a_{ij})$ such that $A \lambda$ has precisely $t$ nonzero rows. If we set $\hat x_{i}=\sum_{j=1}^na_{ij}x_{j} , 1\leq i \leq n$, then $\hat \lambda=A\lambda$ by Lemma \ref{lem:lb}, and it follows that the set $\Big\{ \hat T_{i}=\ad\hat x_{i}\,\Big \vert \,1 \leq i \leq n\Big\}$ contains only $t$ non-nilpotent matrices, or equivalently, the basis $\{\hat x_{i}: 1\leq i\leq n\}$ has $t$ nilpoten elements. This  contradicts with the fact that $x_{k+1}, x_{k+2},\hdots, x_{n}$ are nil-independent.
\end{proof}

It was shown in \cite{Sn} that for a solvable Lie algebra $\li$ one has 
\[\dim \li/nil(\li)\leq \dim nil(\li)-\dim [nil(\li), nil(\li)].\] 
Adding $\dim \li/nil(\li)$ to both sides above one sees in particular that \[2\dim \li/nil(\li)\leq \dim \li.\]
Thus the next corollary follows.
\begin{corollary}
Given a solvable Lie algebra $\li$ we have \[\rank \lambda\le \dim nil(\li)-\dim [nil(\li), nil(\li)].\]
In particular, one has $\rank  \lambda\le \frac{\dim \li}{2}.$
\end{corollary}

Clearly, for a solvable Lie algebra $\li$, its {\cp} $Q(z)=z_{0}^n$ with respect to some basis if and only if the spectral matrix $\lb=0$ with respect to the same basis. In view of Proposition \ref{prop:rank} this is so if and only if $\li=nil(\li)$. Hence the following fact is immediate.
\begin{corollary}\label{cor:nil}
A $n$-dimensional Lie algebra $\li$ is nilpotent if and only if its {\cp} $Q(z)=z_{0}^n$.
\end{corollary}

\section{Spectral invariants}

By Levi decomposition theorem every Lie algebra is a semidirect sum of a semisimple Lie subalgebra with a solvable ideal. Since complex simple Lie algebras are classified by Dynkin diagram, the classification problem for Lie algebras rested on a classification of solvable ones. But this seems to be quite challenging even at low dimensions, and we refer the readers to \cite{GK,Gr,Mu,Mu2,PSW,PZ} for more information on this effort.
One fruitful alternative is to classify solvable Lie algebras with the
same nilradical, or in other words to classify different extensions of a given nilpotent Lie algebra. Details on this approach can be found in \cite{Sn,SW1,SW2} and the references therein. In this section we give examples to show how {\cp} and spectral matrix may assist in the classification scheme.

If $\li$ is nilpotent then its {\cp} is a power of $z_0$ by Corollary \ref{cor:nil} and the spectral matrix is $0$. This fact provides convenience for the study of solvable extensions of nilpotent algebras.
Proposition \ref{prop:rank} indicates that for a solvable Lie algebra $\li$ the rank of its spectral matrix 
is invariant with respect to change of basis and isomorphism of Lie algebras. It is quite tempting to ask if there may be other invariants linked with characteristic polynomial and the spectral matrix. A close observation of Lemma \ref{lem:iso} and factorization (\ref{eq:factors}) suggestes that the number of distinct factors, which we denote by $k(\li)$, in the {\cp} $Q_\li(z)$ is invariant with respect to change of basis and isomorphism. Given a finite set $S$ of complex numbers, we let $|S|$ denote the cardinality of $S$. 
\begin{proposition}\label{prop:kl}
Let $\li$ be a $n$-dimensional solvable Lie algebra with basis $\{x_1,x_2,\cdots,x_n\}$. Then
$k(\li)\geq \max \{ \Big \vert \,\sigma(\ad x_{i}) \, \Big\vert \,i=1,2,\hdots,n \}$. 
\end{proposition}
\begin{proof}
Given the factorization
\[Q_{\li}(z)=\prod_{j=1}^{n}\left(z_{0} +\sum _{i=1}^n \lambda_{ij}z_{i}\right),\]
we have $\lambda=(\lambda_{ij})_{n \times n},$ and $\sigma(\ad x_{i})=\{ \lambda_{ij} \Big \vert\,\,j=1,2,\hdots,n \}.$ We set 
\[t=\max \{ \Big \vert \,\sigma(\ad x_{i}) \, \Big\vert \,i=1,2,\hdots,n \}\]
and assume $t=\Big \vert \,\sigma(\ad x_{m}) \, \Big\vert$ for some $1\leq m\leq n$. Further, we assume
$\lb_{mj_1}, ..., \lb_{mj_t}$ are the distinct eigenvalues in $\sigma(\ad x_m)$. Then the factors 
\[z_0+\lb_{1j}z_1+\cdots + \lb_{mj}z_m+\cdots +\lb_{nj}z_n\] in the factorization above are distinct for $j=j_1, ..., j_t$ because the coefficients of the variable $z_m$ are distinct. This imply that $Q_\li$ has at least $t$ distinct factors, i.e., we have $k(\li)\geq t$.
\end{proof}
The fact that $k(\li)$ is an invariant for $\li$ is not so obvious.	It is a tempting question whether one may determine $k(\li)$ by other known invariants of $\li$. The next example indicates that this might not be easy.

\begin{exmp}\label{ex:lab}
Let $\li_{a,b}=span \Big\{x_{1},x_{2},x_{3} \Big\}$ be a complex Lie algebra with non-zero structure constants specified by the brackets:
$[x_{3},x_{1}]=x_{2},\,\,[x_{3},x_{2}]=ax_{1}+bx_{2}$, where $b\neq 0$. Then\\
\[T_{1}=
\left(\begin{matrix}
0 & 0 & 0  \\
0 & 0 & -1  \\
0 & 0 & 0 
\end{matrix}\right),\,\,
T_{2}=
\left(\begin{matrix}
0 & 0 & -a  \\
0 & 0 & -b  \\
0 & 0 & 0 
\end{matrix}\right),\,\,
T_{3}=
\left(\begin{matrix}
0 & a & 0  \\
1 & b & 0  \\
0 & 0 & 0 
\end{matrix}\right),\] and
\[Q_{\li_{a,b}}(z)=z_{0}(z_{0}+\lb_{32}z_{3})(z_{0}+\lb_{33}z_{3}),\] where \[\lb_{32}=\frac{-b+\sqrt{b^2+4a}}{2}, \ \lb_{33}=\frac{-b-\sqrt{b^2+4a}}{2}.\]
Therefore, $\sigma(T_{1})=\Big\{0\Big\}=\sigma(T_{2})$ and 
$\sigma(T_{3})=\{0 ,\lambda_{32}, \lambda_{33}\}$. Hence the spectral matrix 
\[\lambda=\left(\begin{matrix}
0 & 0 & 0 \\
0 & 0 & 0 \\
0 & \lambda_{32} & \lambda_{33} 
\end{matrix}\right),\]
 and $k(\li_{a,b})=2$ when $b^2+4a= 0$ and $k(\li_{a,b})=3$ when $b^2+4a\neq 0$.

Suppose $\li_{a,b}$ is isomorphic to $\li_{a',b'}=span\{x_1', x_2', x_3'\}$ then by Lemma \ref{lem:lb} and the comments after there exists a matrix $A=B^{-1}\in GL_3$, such that 
\[\left(\begin{matrix}
0 & 0 & 0 \\
0 & 0 & 0 \\
0 &  {\lambda'_{32}} &  {\lambda'_{33}}
\end{matrix}\right)=\lambda'=A\lambda=\left(\begin{matrix}
0 & a_{13}\lambda_{32} & a_{13}\lambda_{33} \\
0 & a_{23}\lambda_{32} & a_{23}\lambda_{33} \\
0 & a_{33}\lambda_{32} &  a_{33}\lambda_{33}
\end{matrix}\right).\]
It follows that ${\lambda'_{32}}=a_{33}{\lambda_{32}} ,\,\, {\lambda'_{33}}=a_{33}{\lambda_{33}}$ and hence $\frac{{\lambda'_{32}}}{{\lambda'_{33}}}=\frac{{\lambda_{32}}}{{\lambda_{33}}}$, i.e., 
\[\frac{-b' + \sqrt{b'+4a'}}{-b' - \sqrt{{b'}^2+4a}}=\frac{-b + \sqrt{b^{2}+4a}}{-b - \sqrt{b^{2}+4a}}=k,\,\, k \in \mathbb{C}.\]
This implies 
\[\frac{b^2 +4a}{b^2}=\frac{(k-1)^2}{(k+1)^2}=\frac{{b'}^2 +4a'}{{b'}^2},\]
and consequently $\frac{a}{b^2}=\frac{a'}{{b'}^2}$ as obtained in \cite{Gr} through Gr\"{o}bner basis computation.
\end{exmp}

Example \ref{ex:lab} indicates that $k(\li_{a,b})$ subtly depends on the parameters $a$ and $b$, and hence it seems to be finer than known algebraic invariants about $\li_{a,b}$.

Interestingly, Theorem \ref{thm:aut} can help to determine the automorphism group $\aut(\li_{a,b})$ as well. If $\phi \in \aut (\li_{a,b})$ with matrix representaion $A=(a_{ij})\in GL_3$, then since the {\cp} $Q_{\li_{a,b}}$ is invariant under change of variables by $A$, in view of the factors of $Q_{\li_{a,b}}(z)$ we see that the variable $z_3$ must be fixed by this change of variables. A direct computation shows $a_{13}=a_{23}=0$ and $a_{33}=1$. Checking on the relations $[\phi (x_i), \phi(x_j)]=\phi([x_i, x_j])$ we obtain
\[a_{21}=aa_{12},\ \ a_{22}=a_{11}+ba_{12},\]
and therefore $\aut(\li_{a,b})$ is equal to the set of complex matrices
\[\left(\begin{matrix}
s & t & 0  \\
at & s+bt & 0  \\
u & v & 1 
\end{matrix}\right),\ \ s^2+bst-{a}t^2\neq 0.\]
It is not hard to see that if the above matrix is unitary then $u=v=0$ and $|a|=1$. Hence $\aut(\li_{a,b})$
contains no unitary elements if $|a|\neq 1$. Example \ref{ex:lab} gives rise to the following generalization.
\begin{corollary}\label{cor:ex}
Let $\mathcal N$ be a $(n-1)$-dimensional nilpotent Lie algebra, where $n\geq 1$, and assume 
$\li=\C x_n\niplus \mathcal N$ and $\li'=\C x'_n\niplus \mathcal N$ are two $1$-dimensional solvable extensions of $\mathcal N$. If $\li$ and $\li'$ are isomorphic then there exists a scalar $t\neq 0$ such that $\sigma(T_n)=t \sigma(T'_n)$.
\end{corollary}
\begin{proof}
Let $S=\{x_1,x_2,\cdots,x_{n-1}\}$ be a basis for $\mathcal N$. Then $S\cup \{ x_n\}$ and $S\cup \{x'_n\}$ are bases for $\li$ and $\li'$, respectively. Their corresponding spectral matrices
$\lb$ and $\lb'$ have nonzero entries only at the bottom row. Suppose the bottom row of $\lb$ is 
$(0, \mu_1, ..., \mu_{n-1})$ and that for $\li'$ is $(0, \mu'_1, ..., \mu'_{n-1})$. If $\phi: \li\to\li'$ is an isomorphism, then Lemma \ref{lem:lb} implies $\lb=B\lb'$ for some $B=(b_{ij})\in GL_n$. Therefore, we have $\mu_j=b_{nn} \mu'_j,\ 1\leq j\leq n-1$, which completes the proof with $t=b_{nn}$.
\end{proof}

\begin{rmk} Corollary \ref{cor:ex} suggests that, up to a nonzero scalar multiple, the spectrum $\sigma(T_n)$ is an invariant for the isomorphism classes of $1$-dimensional extensions of a $(n-1)$-dimensional nilpotent Lie algebra.
\end{rmk}
We use an example to show the usefulness of Corollary \ref{cor:ex}.
\begin{exmp}\label{ex:ab}
Consider solvable Lie algebra $A_{a,b}=span\Big\{x_{1},x_{2},x_{3},x_{4} \Big \}$, with brackets $[x_{1},x_{4}]=ax_{1}$, $[x_{2},x_{4}]=bx_{2}-x_{3}$, $[x_{3},x_{4}]=x_{2}+bx_{3}$, $a > 0$. This algebra is of type $L_{4,6}$ according to \cite{Mu}, and it is a $1$-dimensional extension of the $3$-dimensional abelian Lie algebras generated by $\{x_1, x_2, x_3\}$. One verifies easily that
\[z_0I+\sum_jz_jT_j=\left(\begin{matrix}
z_0-az_4 & 0 & 0 & az_1\\
0 & z_0-bz_4 & -z_4 & bz_2+z_3\\
0 & z_4 & z_0-bz_4 & -z_2+bz_3\\
0 & 0 & 0 & z_0
\end{matrix}\right),\]
and hence the {\cp} \[Q(z)=z_{0}(z_{0}-az_{4})(z_{0}-(b+i)z_{4})(z_{0}-(b-i)z_{4}),\]
which indicates that $\sigma(T_4)=\{0, -a, -b-i, -b+i\}$. If $A_{a',b'}$ is isomorphic to $A_{a,b}$ then by Corollary \ref{cor:ex} there exists a nonzero scalar $k$ such that
\[ka'=a,\ k(b'+i)=b+i,\ k(-b'+i)=-b+i.\]
Adding the last two equations one obtains $k=1$ and hence $a=a', b=b'$. This shows that $A_{a',b'}$ is isomorphic to $A_{a,b}$ only if $a=a', b=b'$.
\end{exmp}

We end this section by a result about unitary elements in the automorphism group. It is a generalization of the observation made before Corollary \ref{cor:ex}.

\begin{corollary}
Suppose $\li$ is a solvable Lie algebra of dimension $n\geq 1$ such that $\dim nil(\li)=n-1$. Then \[\aut(\li)\cap U(n)= \aut(nil(\li))\cap U(n-1)\oplus 1.\]
\end{corollary}
\begin{proof}
Let $\{x_1,x_2,\cdots,x_{n-1},x_{n}\}$ be a basis of $\li$ such that $x_i\in nil(\li),\ 1\leq i\leq n-1,$ and $x_{n}\notin nil(\li)$. Then the {\cp} of $\li$ is of the form
\[Q(z)=\prod_{j=1}^n(z_0+\mu_jz_n),\]
where $\mu_j, 1\leq j\leq n$ are the eigenvalues of $T_n$ which are not all $0$. If $A=(a_{ij})\in \aut(\li)$, then the fact $Q(z)=Q(zA)$ by Lemma \ref{lem:iso} implies that $a_{in}=0,\ 1\leq i\leq n-1$ and $a_{nn}=1$. If in addition $A$ is a unitary then one also has $a_{ni}=0,\ 1\leq i\leq n-1$, i.e., $A$ is a diagonal block matrix $A'\oplus 1$ for some unitary $A'\in \aut(nil(\li))$. 
\end{proof}

\section{Eigen-variety and Poincar\'{e} polynomial}

Given a Lie algebra $\li$ of dimension $n\geq 2$ with basis $S=\{x_1,x_2,\cdots ,x_n\}$, its {\cp} $Q_{\li}(z)$ is a homogeneous polynomial in $n+1$ variables with degree $n$. Its zero variety $V_\li:=\{z\in \C^{n+1} : Q_\li(z)=0\}$, which we call the eigen-variety of $\li$ with respect to the basis $S$, is a finite dimensional case of the projective spectrum mentioned in the introduction. It is of great interest to see how the algebraic structure of $\li$ is encoded in the eigen-variety $V_\li$. First, since by Lemma \ref{lem:z0} the {\cp} for all finite dimensional Lie algebras contains the factor $z_0$, the hyperplane $\{z_0=0\}$ is present in all such $V_\li$. Hence it makes sense to exclude this hyperplane in the discussion. Assume $Q_\li(z)=z_0^kp(z)$, where $k\geq 1$ and $z_0$ is not a factor of polynomial $p(z)$ (of degree $n-k$). Then we consider the variety $V^*_\li=\{p(z)=0\}$. Formula (\ref{eq:trans}) indicates that if $\hat{S}=\{\hat{x}_1,\hat{x}_2,\cdots,\hat{x}_n\}$ is another basis for $\li$ and 
$\hat{x}_i=b_{ij}x_j$ is the change of basis, then the characteristic polynomials with respect to the  two bases satisfy $\hat{Q}(z)=Q(zB)$. This shows that the two eigen-varieties $V^*_\li$ and $V^*_{\hat{\li}}$ differ only by an invertible linear transformation. Therefore the topology of the complement $(V^*_\li)^c$ is invariant with respect to the change of basis or isomorphism of Lie algebras. In view of this fact we shall make the following definition.

\begin{definition}
For a finite dimensional Lie algenra $\li$, the numbers $b_j=\dim H^j((V^*_\li)^c, \C),$ $j\geq 0$ are called the Betti numbers for $\li$. And the Poincar\'{e} polynomial for $\li$ is defined as
$P_\li(t)=\sum_{j}b_jt^j.$
\end{definition}
A clarification is needed. If $\dim \li=n$ then the open set $(V^*_\li)^c$, being the complement of the zero variety of $Q_\li$, is a {\em domain of holomorphy}. Hence its de Rham cohomology are generated by holomorphic forms on $(V^*_\li)^c$ (\cite{Kr, Ra}). Therefore, although $(V^*_\li)^c$ is a domain in ${\mathbb R}^{2n+2}$, the Betti numbers $b_j$ are 0 for $j>n+1$. This shows that the degree of $P_\li(t)$ is less than or equal to $1+\dim \li$. 

In the case when $\li$ is solvable, Hu-Zhang theorem indicates that
its {\cp} $Q_{\li}(z)$ is a product of linear factors. Hence $V_\li$ is a union of $n$ hyperplanes, counting multiplicity. Such union is called a hyperplane arrangement. An important theorem is needed to proceed. Given vectors $v_1,v_2,\cdots,v_k$ in $\C^{n+1}$, we consider the union of hyperplanes $M:=\cup_{j=1}^kH_j$, where $H_j=\{z\in \C^{n+1}: \langle z, v_j\rangle=0\},\ \ 1\leq j\leq k.$
It is of great interest to study the topology of the complement $M^c=\C^{n+1}\setminus M$ which obviously depends on the relative position of the hyperplanes $H_j$.
Observe that the $1$-forms $\omega_j=\frac{d\langle z, v_j\rangle}{\langle z, v_j\rangle}, j=1,\cdots, k$ are well-defined and non-exact on the complement $M^c$, and hence it is a nontrivial element in the first de Rham cohomology $H^1(M^c, \C)$. Remarkably, a theorem conjectured by Arnold \cite{Ar} and proved by Brieskorn \cite{Br} states that the full de Rham cohomology $H^*(M^c, \C)$ (as a graded algebra with wedge product) is in fact generated by the $1$-forms $\omega_j, 1\leq j\leq k.$ For more information on hyperplane arrangement we refer readers to \cite{OT}. In this section we shall take a look at the Poincar\'{e} polynomial for the arrangement $V^*_\li$, where $\li$ is solvable. We start with an example.

\begin{exmp}\label{ex:poin}
Consider the $L_{4,6}$ type Lie algebra $\li=A_{a,b}$ in Example \ref{ex:ab}, where the {\cp} is \[Q(z)=z_{0}(z_{0}-az_{4})(z_{0}-(b+i)z_{4})(z_{0}-(b-i)z_{4}).\]
Then $V^*_\li$ is the union of the hyperplanes
\[H_1=\{z_0-az_4=0\},\ H_2=\{z_0-(b+i)z_4=0\},\ H_3=\{z_0-(b-i)z_4=0\}.\]
And by Arnold-Brieskorn theorem the de Rham cohomology of $(V^*_\li)^c$ is generated by the three $1$-forms
\[\omega_1=\frac{dz_{0}-a dz_{4}}{z_{0}- a z_{4}},\ \ \omega_2=\frac{dz_{0}-(b + i) dz_{4}}{z_{0}-(b + i)z_{4}},\ \ \omega_3=\frac{dz_{0}-(b - i) dz_{4}}{z_{0}-(b - i)z_{4}},\]
which means 
\[H^1=span\{\omega_j : j=1, 2, 3\},\ \ H^2=span\{\omega_i\wedge \omega_j : j, k=1, 2, 3\}.\]
It is clear that $H^0=\C$ and $H^j=\{0\}$ for $j\geq 3$. Then the Poincar\'{e} polynomial has two possible forms. 

1) If $a$ is not equal to $b\pm i$, then $b_1=3$ and
direct computation shows that $H^2=\frac{1}{Q^*(z)} span\{z_0, z_4\}dz_0\wedge dz_4$ implying $b_2=2$. Hence the Poincar\'{e} polynomial $P_\li(t)= 1+3t+2t^2$. 

2) In the case $a=b\pm i$, then $\omega_1=\omega_2$ or $\omega_1=\omega_3$, and it follows that
$P_\li(t)= 1+2t+t^2$.
\end{exmp}

We pause here to mention that the cohomology of $(V^*_\li)^c$ is quite different from the cohomology theory on Lie algebras as treated in \cite{Di,GK}. The following fact follows immediate from Corollary \ref{cor:nil}.

\begin{corollary}
A solvable Lie algebra $\li$ is nilpotent if and only the Poincar\'{e} polynomial $P_\li$ is the constant $1$.
\end{corollary} 
On the other hand, nilpotent Lie algebras can have nontrivial cohomology (\cite{GK}). Furthermore, it is not difficult to see that the first Betti number coincides with the number of hyperplanes in the arrangement.
\begin{proposition}
If $\li$ is a solvable Lie algebra then $b_1(\li)=k(\li)-1$.
\end{proposition}
\begin{proof}
Suppose ${v}_{1},{v}_2, \hdots , {v}_{m}$ are distinct nonzero vectors in $\mathbb{C}^{n}$, where $m \leq n.$ As in Section 2.1 we write $z=(z_0, z')$ and claim that the $1$-forms
\[\omega_j=\frac{dz_0+d\langle z', v_j\rangle}{z_0+\langle z', v_j\rangle}=d\log (z_0+\langle z', v_j\rangle), \ \ 1\leq j\leq m\] are linearly independent. Indeed, if there are constants $\alpha_1,\alpha_2,\cdots,\alpha_m$ such that 
\[0=\sum_j\alpha_j\omega_j=d\log \prod_j (z_0+\langle z', v_j\rangle)^{\alpha_j},\ \forall z\in (V^*_\li)^c,\]
then $\prod_j (z_0+\langle z', v_j\rangle)^{\alpha_j}$ must be a constant. Since the linear functions $z_0+\langle z', v_j\rangle, 1\leq j\leq m$ are distinct, it follows that each $\alpha_j=0$ for each $j$.

If $\li$ is solvable then its {\cp} has $k(\li)-1$ distinct factors other than the factor $z_0$, and by the foregoing arguments the $1$-forms associated with these factors are linearly independent. This implies that 
$b_1(\li)=k(\li)-1$.
\end{proof}
The following question seems natural.
\begin{question}
Can the Betti numbers $b_j(\li),\ j\geq 2$ be determined by known invariants of $\li$?
\end{question}

Now we try to determine the degree of the Poincar\'{e} polynomial. 
\begin{proposition}\label{prop:deg}
For any finite dimensional solvable Lie algebra $\li$ we have \[\deg P_\li\leq \dim \li/nil(\li)+1.\]
\end{proposition}
\begin{proof}
Let $\li$ be a solvable Lie algebra with basis $\Big \{x_{1},x_{2},\hdots,x_{n} \Big\}$ and the associated spectral matrix $\lb$. Then by Proposition \ref{prop:rank} we have $\rank\lambda=\dim \li/nil(\li)$. We assume $\rank\lambda =m$ and claim that there exists a basis $\Big \{\hat x_{1},\hat x_{2},\hdots,\hat x_{n} \Big\}$ for $\li$ such that the characteristic polynomial  $\hat{Q}_\li$ with respect to this basis has precisely $m+1$ variables. In fact, 
let $L$ be an invertible matrix of row operations on $\lb$ such that $L\lb$ has precisely $m$ linearly independent rows. 
Set $B=L^{-1}=(b_{ij})_{n \times n}$, and define $\hat x_{i}=\sum_{j=1}^nb_{ij}x_{j}.$ Then by Lemma \ref{lem:lb} and its proof we have  $\lb=B\hat \lambda$, and hence $\hat{\lb}=L\lb$. Therefore,
\begin{equation}\label{eq:qhat}
\hat{Q}(z)=\det(z_{0}I +\sum_{j=1}^n z_{j}\ad\hat x_{j})=\prod_{j=1}^n \left(z_{0}+\sum_{i=1}^n \hat \lambda_{ij} z_{i}\right).
\end{equation}
Since the entries in the $i$-th row of $\hat{\lb}$ are the coefficients of $z_i$ in (\ref{eq:qhat}), the number of variables $z_{i}$'s in (\ref{eq:qhat}) coincides with the number $m$ of nonzero rows in $\hat{\lb}$. Counting the variable $z_0$, the {\cp} $\hat{Q}$ therefore has precisely $m+1$ variables. Since the eigen-variety $\hat{V}_\li:=\{\hat{Q}(z)=0\}$ is situated in $\mathbb{C}^{m+1}$ and its complement $\hat{V}^c_\li$ is a domain of holomorphy, one naturally has $b_j=\dim H^j(\hat{V}^c_\li, \C)=0$ for all $j> m+1$. This completes the proof.
\end{proof}

In particular, if $\li$ is a solvable $1$-dimensional extension of a nilpotent Lie algebra then $P_\li$ is linear or quadratic. Example \ref{ex:poin} shows that the inequality in Proposition \ref{prop:deg} is sharp.

\section{Concluding remarks}

The idea of {\cp} for Lie algebras seems fairly intriguing. This paper aims to identify some invariants for Lie algebras based on their {\cp}s. The results here appear to offer a point of view from which algebraic geometry, topology and spectral theory provide new and natural tools for the study of Lie algebras. Application to the classification scheme of solvable Lie algebras, as attempted in Chapter 5 and 6, is rather preliminary at this point. Nevertheless, we hope this paper may help to lay a foundation for more in-depth explorations on the characteristic polynomial of Lie algebras.


\begin{thebibliography}{1}

\bibitem{Ar} V. I. Arnold, {\em the Cohomology ring of the colored braid group}, Mat. Zametki, Volume 5 (1969), Issue 2, 227-231.

\bibitem{BCY} J. Bannon, P. Cade and R. Yang, {\em On the Spectrum of Operator-valued Entire Functions}, Illinois J. of Mathematics 55 No.4 (2011).

\bibitem{Br} E. Brieskorn, {\em Sur les groupes de tresses}, In: Seminaire Bourbaki 1971/72,
Lecture Notes in Math. 317 (1973), Springer-Verlag, 21-44.

\bibitem{CCD} Z. Chen, X. Chen and M. Ding, {\em On the characteristic polynomial of $\frak{sl}(2, \mathbb F)$}, Linear Alg. Appl. 579 (2019), 237-243.

\bibitem{CY} P. Cade and R. Yang, {\em Projective Spectrum and Cyclic Cohomology}, J. Funct. Analy. Vol. 265 (2013), No. 9, 1916-1933.

\bibitem{CST} \v{Z}. \v{C}u\v{c}kvi\v{c}, M. Stessin and A. Tchernev, {\em Determinantal hypersurfaces and representations of Coxeter groups}, arxiv 1810.12893v1.

\bibitem{CSZ} I. Chagouel, M. Stessin and K. Zhu, {\em Geometric Spectral Theory for Compact Operators}, Trans. Amer. Soc. 368 (2016), No. 3, 1559-1582.

\bibitem{Cu} C. Curtis, {\em Representation Theory of Finite Groups: from Frobenius to Brauer}, Math. Intelligencer 14 (1992), 48-57.

\bibitem{Cu2} C. Curtis, {\em Pioneers of Representation Theory: Frobenius, Burnside, Schur, and Brauer}, History of Mathematics, Providence, R.I.: American Mathematical Society, 2003.

\bibitem{De} R. Dedekind, {\em Gesammelte Mathematische Werke}, Vol. II. Chelsea, New York, 1969.

\bibitem{Di} J. Dixmier, {\em Cohomologie des alg\`{e}bras de Lie nilpotentes}, Acta Sci. Math. (Szeged) 16 (1955), 246-250.

\bibitem{Di21} L. E. Dickson, {\em Determination of all General Homogeneous Polynomials Expressible as Determinants with Linear Elements}, Trans. Amer. Math. Soc. 22 (1921), No. 2,167 - 179.

\bibitem{Di75} L. E. Dickson, {\em An Elementary Exposition of Frobenius Theory of Group Characters and Group Determinants}, Ann. of Math. 4 (1902), 25-49; Mathematical Papers, Vol. II. Chelsea, New York,1975, 737-761.

\bibitem{DY} R. G. Douglas and R. Yang, {\em Hermitian Geometry on Resolvent set (I)}, Operator theory, operator algebras, and matrix theory, 167-183, Oper. Theory Adv. Appl., 267, Birkh\"{a}user/Springer, Cham, 2018.

\bibitem{FH} W. Fulton and J. Harris, {\em Representation theory: a first course}, Undergraduate Text in Mathematics, Springer-Verlag, New York, 1991.

\bibitem{Fr} F. G. Frobenius, {\em \"{U}ber vertauschbare Matrizen, Sitzungsberichte der K\"{o}niglich Preussischen}, Akademie der Wissenschaften zu Berlin (1896) 601-614; Gesammelte Abhandlungen, Band II Springer-Verlag, New York, 1968, 705-718.

\bibitem{FS} E. Formanek and D. Sibley, {\em The Group Determinant Determines the Group}, Proc. A.M.S Vol 112 (1991), No. 3, 649-656.

\bibitem{GK} M. Goze and Y. Khakimdjanov, {\em Nilpotent Lie algebras}, Mathematics and Its Applications 361, Springer-Sci.+Bus. Media, B.V., 1996.


\bibitem{Gr} W. de Graaf, {\em Classification of solvable Lie algebras}, Experiment. Math. 14 (2005), No. 1, 15-25.

\bibitem{GOY} B. Goldberg and R. Yang, {\em Projective spectrum and spectral dynamics}, arxiv 2002.09791.

\bibitem{GY} R. Grigorchuk and R. Yang, {\em Joint Spectrum and the Infinite Dihedral Group}, Proc. of the Steklov Institute of Math., 2017, Vol. 297, 145-178.

\bibitem{Kr} S. Krantz, {\em Function theory of several complex variables: 2nd ed.}, AMS Chelsea Publishing Vol. 340, 1992.

\bibitem{Hu} J. E. Humphreys, {\em Introduction to Lie algebras and representation theory}, Springer-Verlag, New York, 1972.

\bibitem{HY18} Z. Hu and R. Yang, {\em On the Characteristic Polynomials of Multiparameter Pencils}, Linear Alg. and its Appli. 558 (2018), 250-263.

\bibitem{HWY} W. He, X. Wang and R. Yang, {\em Projective Spectrum and Kernel Bundle(II)}, J. of Operator Theory 78 (2017), No. 2, 417-433.

\bibitem{HZ} Z. Hu and B. Zhang, {\em Determinants and characteristic polynomials of Lie algebras}, Linear Alg. Appl. 563 (2019), 426-439.

\bibitem{MQW17} T. Mao, Y. Qiao and P. Wang, {\em Commutativity of Normal Compact Operator via Projective Spectrum}, Proc. A.M.S 146 (2017), No. 3, 1165-1172.

\bibitem{Mu} G. M. Mubarakzjanov, {\em Classification of real structures of Lie algebras of fifth order}, Izv. Vys\v{s}. U\v{c}ebn. Zaved. Matematika, 1963 (3 (34)): 99-106 (Russian).

\bibitem{Mu2} G. M. Mubarakzjanov, {\em Classification of solvable Lie algebras of sixth order with a non- nilpotent basis element}, Izv. Vys\v{s}. U\v{c}ebn. Zaved. Matematika, 1963 (4 (35)): 104-116 (Russian).

\bibitem{Os} T. Ostrowski, {\em A note on semidirect sum of Lie algebras}, Discussiones Mathematicae General Algebra and Applications 33 (2013), 233-247.

\bibitem{OT} P. Orlik and H. Terao, {\em Arrangements of Hyperplanes}, Grundlehren der Mathematischen Wissenschaften [Fundamental Principles of Mathematical Sciences] 300, Berlin: Springer-Verlag, 1992.

\bibitem{PSW} J. Patera, R. T. Sharp, P. Winternitz and H. Zassenhaus, {\em Invariants of real low dimension Lie algebras}, J. Math Phys 17 (1976), No. 6, 986-994.

\bibitem{PZ} J. Patera and H. Zassenhaus, {\em Solvable Lie algebras of dimension {$\leq 4$} over perfect fields}, Linear Alg. Appl. 142 (1990), 1-17.

\bibitem{Ra} M. Range, {\em Holomorphic functions and integral representations in
several complex variables}, Graduate Texts in Mathematics, Springer-Verlag, New York, 1986.

\bibitem{Sn} L. \v{S}nobl, {\em On the structure of maximal solvable extensions and of Levi extensions of nilpotent algebras}, J. Phys. A: Math. Theor. 43 (2010), No. 50, 505202 (17pp).

\bibitem{SW1} L. \v{S}nobl and P. Winternitz, {\em A class of solvable Lie algebras and their Casimir invariants}, J. Phys. A: Math. Gen. 38 (2005), 2687-2700.

\bibitem{SW2} L. \v{S}nobl and P. Winternitz, {\em All solvable extensions of a class of nilpotent Lie algebras of dimension $n$ and degree of nilpotency $n-1$},  J. Phys. A: Math. Theor. 42 (2009), No. 10, 105201 (16pp).

\bibitem{SYZ} M. Stessin, R. Yang and K. Zhu, {\em Analyticity of a Joint Spectrum and a Multivariable Analytic Fredholm Theorem}, New York J. Math. 17A (2011), 39-44.


\bibitem{Ya} R. Yang, {\em Projective Spectrum in Banach Algebras}, J. Topol. and Analy. 1 (2009), No. 3, 289-306.

\end{thebibliography}
\end{document}